\documentclass[a4paper,12pt]{amsart}

\usepackage{amssymb}
\usepackage{graphicx}
\usepackage{float}
\usepackage{amsmath}
\usepackage{array}
\usepackage{multirow}
\usepackage{mathrsfs}
\usepackage{textcomp}
\usepackage[bottom]{footmisc}
\usepackage[dvipsnames]{xcolor}
\usepackage{cite}
\usepackage{soul}
\usepackage[all,cmtip]{xy}
\usepackage{amscd}

\newcolumntype{M}[1]{>{\raggedright}m{#1}}

\DeclareMathAlphabet{\mathpzc}{OT1}{pzc}{m}{it}

\newtheorem{theorem}{Theorem}[section]
\newtheorem{lemma}[theorem]{Lemma}

\theoremstyle{definition}
\newtheorem{definition}[theorem]{Definition}

\theoremstyle{remark}
\newtheorem{remark}[theorem]{Remark}

\numberwithin{equation}{section}

\allowdisplaybreaks[4]

\begin{document}

\title{On Milnor--Orlik's theorem and admissible\\ simultaneous good resolutions}

\author{Christophe Eyral and Mutsuo Oka}

\address{C. Eyral, Institute of Mathematics, Polish Academy of Sciences, ul. \'Sniadeckich 8, 00-656 Warsaw, Poland}  
\email{cheyral@impan.pl} 
\address{M. Oka, Emeritus Professor, Tokyo Institute of Technology, 152-8551 Ohokayama, Meguro-ku, Tokyo, Japan}   
\email{okamutsuo@gmail.com}

\thanks{}

\subjclass[2020]{14B05, 32S05, 32S25, 32S45.}

\keywords{Isolated surface singularity, weighted homogeneous polynomial, monodromy zeta-function, Milnor number,  Newton non-degeneracy/degeneracy, resolution of singularities.}

\begin{abstract}
Let $f$ be a (possibly Newton degenerate) weighted homogeneous polynomial defining an isolated surface singularity at the origin of $\mathbb{C}^3$, and let $\{f_s\}$ be a generic deformation of its coefficients such that $f_s$ is Newton non-degenerate for $s\not=0$. We show that there exists an \emph{``admissible'' simultaneous good resolution of the family of functions $f_s$ for all small $s$}, including $s=0$ which corresponds to the (possibly Newton degenerate) function $f$. 
As an application, we give a new geometrical proof of a weak version of the Milnor--Orlik theorem that asserts that the monodromy zeta-function of $f$ (and hence its Milnor number) is completely determined by its weight, its weighted degree and its Newton boundary.
\end{abstract}

\maketitle

\markboth{Christophe Eyral and Mutsuo Oka}{On Milnor--Orlik's theorem and admissible simultaneous good resolutions}  

%%%%%%%%%%%%%%%%%%%%%%%%%%%%%%%%%%%%%%%%%%%%%%%%%%%%%%%%%%%%%%%%%%%%%%%%%%%%%%%%%%%%%%%%%%%%%
\section{Introduction}\label{intro}
%%%%%%%%%%%%%%%%%%%%%%%%%%%%%%%%%%%%%%%%%%%%%%%%%%%%%%%%%%%%%%%%%%%%%%%%%%%%%%%%%%%%%%%%%%%%%

\subsection{The Milnor--Orlik result}\label{MO-result}
Let $p_1,\ldots,p_n$ be positive integers with $\gcd(p_1,\ldots,p_n)=1$, and let $P={}^t(p_1,\ldots,p_n)$ be the corresponding primitive weight vector. In the space of complex polynomials in $n$ variables $z_1,\ldots,z_n$, consider the subspace $\mathcal{W}(P,d)$ of \emph{weighted homogeneous} polynomials of weighted degree $d$ with respect to $P$ (in short, of $P$-degree $d$) which have an \emph{isolated} critical point at the origin $\mathbf{0}\in\mathbb{C}^n$. 
In \cite{MO}, Milnor and Orlik showed the following theorem.

\begin{theorem}[Milnor--Orlik]\label{mt1}
For any $f\in \mathcal{W}(P,d)$, the monodromy zeta-function $\zeta_{f,\mathbf{0}}(t)$ of $f$ at $\mathbf{0}$ is completely determined by~$P$ and $d$. More precisely, for each $1\leq i\leq n$, write $d/p_i=u_i/v_i$ with $u_i,v_i\in\mathbb{N}$ and $(u_i,v_i)=1$. Then the divisor $\mbox{\emph{div}}(\tilde \zeta_{f,\mathbf{0}}(t))$ of the reduced monodromy zeta-function $\tilde \zeta_{f,\mathbf{0}}(t):=\zeta_{f,\mathbf{0}}(t)\cdot (t-1)$ is given by
\begin{equation*}
\mbox{\emph{div}}(\tilde \zeta_{f,\mathbf{0}}(t))=(-1)^n\prod_{i=1}^n \bigg(\frac{1}{v_i}\Lambda_{u_i}-1\bigg).
\end{equation*}
In particular, the Milnor  number $\mu_{f,\mathbf{0}}$ of $f$ at $\mathbf{0}$ is given by
\begin{equation*}
\mu_{f,\mathbf{0}}=\prod_{i=1}^n \bigg(\frac{d}{p_i}-1\bigg)=\prod_{i=1}^n \bigg(\frac{u_i}{v_i}-1\bigg).
\end{equation*}
\end{theorem}

As in \cite{MO}, $\Lambda_{u_i}:=1 + \langle \xi_i\rangle + \cdots +\langle \xi_i^{u_i-1}\rangle$  is the divisor of the polynomial $t^{u_i}-1$ as an element of the integral group ring $\mathbb{Z}\mathbb{C}^*$, and $1:=\langle 1\rangle$ is the divisor of the polynomial $t-1$. Here, $\xi_i:=\exp(2\pi \sqrt{-1}/u_i)$.

Now, consider two (not necessarily weighted homogeneous) polynomials $f$ and $g$ with an isolated singularity at $\mathbf{0}$ and with the same Newton boundary $\Gamma$. By a theorem of Kouchnirenko \cite{K}, if both $f$ and $g$ are Newton non-degenerate, then $\mu_{f,\mathbf{0}}=\mu_{g,\mathbf{0}}$. Indeed, under the non-degeneracy assumption, the Kouchnirenko theorem asserts that the Milnor numbers $\mu_{f,\mathbf{0}}$ and $\mu_{g,\mathbf{0}}$ coincide with the Newton numbers $\nu_{f,\mathbf{0}}$ and $\nu_{g,\mathbf{0}}$ of $f$ and $g$ at $\mathbf{0}$, respectively, and since the latter depend only on the Newton boundary $\Gamma$, they must be equal.
On the other hand, if $f$ or $g$ fails to be Newton non-degenerate, then, in general, $\mu_{f,\mathbf{0}}\not=\mu_{g,\mathbf{0}}$. For curves (i.e., $n=2$), this is systematically the case. Indeed, if our polynomials $f$ and $g$ are curves with $f$ Newton degenerate and $g$ Newton non-degenerate, then the theorem of Kouchnirenko shows that $\mu_{g,\mathbf{0}}=\nu_{g,\mathbf{0}}=\nu_{f,\mathbf{0}}$. However, by \cite{P,GBLP}, the Newton degeneracy of $f$ implies that $\nu_{f,\mathbf{0}}<\mu_{g,\mathbf{0}}$. (Note that a weighted homogeneous curve with an isolated singularity at $\mathbf{0}$ cannot be Newton degenerate.)

The last part of the Milnor--Orlik theorem says that if, in addition for our polynomials $f$ and $g$ to have an isolated singularity at $\mathbf{0}$ and the same Newton boundary, we also assume that they are weighted homogeneous for a given weight $P$ and a given $P$-degree $d$, then this change in Milnor number cannot occur. Indeed, consider a generic deformation $\{f_s\}$ of the coefficients of $f$ such that $f_s$ is Newton non-degenerate for $s\not=0$. Then the $f_s$'s are weighted homogeneous of $P$-degree $d$ and they have an isolated singularity at $0$, so that the family $\{f_s\}$ has a \emph{uniform stable radius} for the Milnor fibrations of the $f_s$'s (see \cite{O4}). Let us denote this radius by $\varepsilon_0$. Then the family of normalized Jacobian mappings
\begin{equation*}%\label{njm}
df_s/\Vert df_s\Vert\colon
\mathbf{z}=(z_1,\ldots,z_n)\mapsto
(\partial_{z_1}f_s(\mathbf{z}),\ldots,\partial_{z_n}f_s(\mathbf{z})) / \Vert (\partial_{z_1}f_s(\mathbf{z}),\ldots,\partial_{z_n}f_s(\mathbf{z})) \Vert,
\end{equation*}
from the sphere $\mathbb{S}_{\varepsilon_0}^{2n-1}$ with radius $\varepsilon_0$ and centre $\mathbf{0}$ to the unit sphere of $\mathbb{C}^n$, induces a homotopy between $df_s/\Vert df_s\Vert$ and $df_0/\Vert df_0\Vert$,
so that these mappings have the same degree, that is, 
\begin{equation*}
\mu_{f_s,\mathbf{0}}=\mu_{f_0,\mathbf{0}}=\mu_{f,\mathbf{0}}
\end{equation*}
 (see \cite[\S 2]{MO} and \cite[\S 7]{M}). The same argument applied to $g$ leads to a family $\{g_s\}$ such that $\mu_{g,\mathbf{0}}=\mu_{g_s,\mathbf{0}}$, and since $f_s$ and $g_s$ have the same Newton boundary $\Gamma$, they have the same Milnor number by Kouchnirenko's theorem. Altogether, $\mu_{f,\mathbf{0}}=\mu_{g,\mathbf{0}}$.

The original motivation of the present paper was to understand the geometry governing this  special feature of isolated singularities of weighted homogeneous hypersurfaces with a given fixed Newton boundary. 

\subsection{Existence of an admissible simultaneous good resolution}
To simplify, we shall only consider polynomials in three variables $z_1,z_2,z_3$.
To investigate the above problem, we show the following theorem about the existence of an ``admissible'' simultaneous good resolution. Let $f\in \mathcal{W}(P,d)$ be a (possibly Newton degenerate) weighted homogeneous polynomial defining an isolated surface singularity at the origin of $\mathbb{C}^3$, and let $\{f_s\}$ be a generic deformation of its coefficients such that $f_s$ is Newton non-degenerate for $s\not=0$. We prove that there exists a regular simplicial cone subdivision of the dual Newton diagram $\Gamma^*(f)$ of $f$ (so-called \emph{admissible subdivision}) such that the corresponding toric modification gives a simultaneous good resolution of the family of functions $f_s$ for all small $s$, including $s=0$ which corresponds to the (possibly Newton degenerate) function $f$ (see Theorem \ref{mp}). We shall call such a resolution an \emph{admissible simultaneous good resolution}.

As an application, we give a new geometrical proof of a weak version of the Milnor--Orlik theorem. More precisely, let $\mathcal{W}(P,d,\Gamma)$ be the subset of $\mathcal{W}(P,d)$ consisting of those polynomials in $\mathcal{W}(P,d)$ whose Newton boundary is $\Gamma$. We show that by combining Theorem \ref{mp} with the famous formula for the monodromy zeta-function of A'Campo \cite[Th\'eor\`eme~3]{A}, the monodromy zeta-function of an element $f\in \mathcal{W}(P,d,\Gamma)$ (and hence its Milnor number) is completely determined by $P$, $d$  and $\Gamma$ (see section~\ref{izf}). 
This new approach of Milnor--Orlik's result is certainly more complicated than the original one by Milnor and Orlik. However, it seems interesting to us because it allows to highlight the geometrical reason why the correspondence $f\mapsto \zeta_{f,\mathbf{0}}(t)$ remains constant on $\mathcal{W}(P,d,\Gamma)$, even if $f$ gets Newton degenerate on some faces of $\Gamma$ \textemdash\ a property which, let us remember, fails if we replace $\mathcal{W}(P,d,\Gamma)$ by the set of polynomials having an isolated critical point at~$\mathbf{0}$ and a given fixed Newton boundary $\Gamma$ but which are not weighted homogeneous.

\subsection{Known related results}
Before getting into the heart of the paper, let us briefly mention some known related results.
The resolution of singularities of complex analytic spaces dates back to Hironaka's founding work \cite{Hironaka}. In \cite{Teissier}, Teissier introduced various notions of simultaneous resolutions (weak, strong, etc) for deformation-families of analytic spaces, and studied the link between such resolutions and certain numerical invariants associated with the family members. For example, he showed that if a family of isolated hypersurface singularities admits a ``strong simultaneous resolution'', then it satisfies the Whitney conditions, that is, by \cite{BS,T}, it is $\mu^*$-constant. This study was continued and developed further by Laufer \cite{L,L2} (see also Shepherd-Barron \cite{S}) who proved theorems giving sufficient conditions for a family of isolated surface singularities to have a simultaneous resolution. For example, Laufer proved that in the case of surfaces, the converse of the above mentioned theorem of Teissier holds true, namely if a family of isolated surface singularities is $\mu^*$-constant, then it admits a strong simultaneous resolution. In \cite{O2}, the second named author showed that a $\mu$-constant family $\{f_s\}$ of Newton non-degenerate surface singularities which induces a ``negligible truncation'' of the Newton boundary of $f_0$ always admits a weak simultaneous resolution.
More recently, in \cite{LMS}, Leyton-Álvarez, Mourtada and Spivakovsky showed that a family of Newton non-degenerate hypersurface singularities admits a simultaneous embedded resolution if and only if it is $\mu$-constant.
Finally, in \cite{FP}, Fukui and P\u{a}unescu gave a result closely related to ours (also allowing Newton degenerate functions) with a different approach based on ``modified analytic trivializations''.

%%%%%%%%%%%%%%%%%%%%%%%%%%%%%%%%%%%%%%%%%%%%%%%%%%%%%%%%%%%%%%%%%%%%%%%%%%%%%%%%%%%%%%%%%%%%%
\section{Existence of an admissible simultaneous good resolution}\label{ESR}
%%%%%%%%%%%%%%%%%%%%%%%%%%%%%%%%%%%%%%%%%%%%%%%%%%%%%%%%%%%%%%%%%%%%%%%%%%%%%%%%%%%%%%%%%%%%%

Assume that $n=3$ and take $f \in \mathcal{W}(P,d)$. Write $f(\mathbf{z})=\sum_{\alpha} c_{\alpha}\, \mathbf{z}^\alpha$, where $\mathbf{z}:=(z_1,z_2,z_3)$, $c_\alpha\in\mathbb{C}$, $\alpha:=(\alpha_1,\alpha_2,\alpha_3)\in\mathbb{N}^3$ and $\mathbf{z}^\alpha:=z_1^{\alpha_1} z_2^{\alpha_2} z_3^{\alpha_3}$. For any subset $I\subseteq \{1,2,3\}$, let $f^I$ denote the restriction of $f$ to the coordinate subspace $\mathbb{C}^I$ defined by $z_i=0$ if $i\notin I$. 
Let $\Gamma_{\! +}(f)$ be the Newton polyhedron of $f$, and let $\Gamma(f)$ be its Newton boundary.
For any weight vector $Q={^t}(q_1,q_2,q_3)$, let $\Delta(Q,f)$ be the face of $\Gamma(f)$ defined by
\begin{equation*}%\label{defdqelq}
\Delta(Q,f):=\Big\{\alpha\in\Gamma_{\! +}(f)\, ;\, \sum_{i=1}^3 \alpha_i q_i=d(Q,f)\Big\}, 
\end{equation*} 
where $d(Q,f)$ is the minimal value of the linear map $\alpha\in\Gamma_{\! +}(f)\mapsto \sum_{i=1}^3 \alpha_i q_i$.
When there is no ambiguity, we write $\Delta(Q)$ and $d(Q)$ instead of $\Delta(Q,f)$ and $d(Q,f)$.

Now, let $U$ denotes a small open neighbourhood of $0\in\mathbb{C}$, and consider a deformation $\{f_s\}_{s\in U}$ of $f$ (i.e., $f_0=f$) obtained from a small analytic deformation of the coefficients of $f$ such that $f_s$ is Newton non-degenerate for $s\not=0$. Put $F(s,\mathbf{z}):=f_s(\mathbf{z})$, $V(F):=F^{-1}(0)\subseteq \mathbb{C}\times \mathbb{C}^3$, and consider a proper map $\pi\colon X\to\mathbb{C}^3$ from an $3$-dimensional (complex) analytic manifold $X$ to $\mathbb{C}^3$. Finally, write $\Pi:=\mbox{id}\times \pi\colon U\times X\to U\times \mathbb{C}^3$.

\begin{definition}\label{def-ugrn}
We say that $\pi\colon X\to\mathbb{C}^3$ is a \emph{simultaneous good resolution of the family of functions $f_s$ for all small $s$} if there exists an open disc $D_\varepsilon(0)\subseteq U$ at the origin $0\in\mathbb{C}$ with radius $\varepsilon>0$ such that the following three conditions hold true:
\begin{enumerate}
\item
the restristed map $\Pi\colon (D_\varepsilon(0)\times X)\setminus \Pi^{-1}(V(F))\to (D_\varepsilon(0)\times\mathbb{C}^3)\setminus V(F)$ is a biholomorphism;
\item
the strict transform $\widetilde{V}(F):=\overline{\Pi^{-1}(V(F)\setminus (\mathbb{C}\times \{0\}))}$ is non-singular;
\item
for any point $(s,x)\in \Pi^{-1}(V(F))$, there exists a system of coordinates at $(s,x)$ such that the pull-back $\Pi^*F:=F\circ\Pi$ is given by a monomial in these coordinates (i.e., $\Pi^{-1}(V(F))$ has only normal crossing singularities).
\end{enumerate}
\end{definition}

These conditions are equivalent to saying that $\Pi$ is a good resolution of the function $F$ in the sense of \cite[Definition (5.1.1)]{O3}.
In particular, this implies that the restricted map $\Pi\colon \{s\}\times X\to \{s\}\times\mathbb{C}^3$ is a good resolution of the function $f_s$ for any $s\in D_\varepsilon(0)$.

\begin{theorem}\label{mp}
There exists a regular simplicial cone subdivision $\Sigma^*$ of the dual Newton diagram $\Gamma^*(f)$ of $f$ such that the corresponding toric modification $\pi\colon X\to\mathbb{C}^3$ gives a simultaneous good resolution of the family of functions $f_s$ for all small $s$.
\end{theorem}

Note that, by our definition of the $f_s$'s, for each $s$ the Newton boundary $\Gamma(f_s)$ and the dual Newton diagram $\Gamma^*(f_s)$ of $f_s$ coincide with $\Gamma(f)$ and $\Gamma^*(f)$, respectively.

For the definitions of regular simplicial cone subdivisions, dual Newton diagram and toric modification, we refer the reader to \cite[Chap.~II, \S 1 and Chap.~III, \S 3]{O3}. A regular simplicial cone subdivision of the dual Newton diagram $\Gamma^*(f)$ is called an \emph{admissible} subdivision (see \cite[Chap.~III, Definition (3.1.1)]{O3}). We shall call the simultaneous good resolution given by Theorem \ref{mp} an \emph{admisible} simultaneous good resolution.

In \cite{FP}, Fukui and P\u{a}unescu gave a closely related result (also allowing Newton degenerate functions) with a completely different approach based on so-called ``modified analytic trivialization''. In general, in this approach, the subdivision is not admissible.
In the special case where $f$ is Newton non-degenerate, Theorem \ref{mp} was first proved in \cite[Chap.~III, Theorem (3.4)]{O3} without assuming that $f$ is weighted homogeneous.

%%%%%%%%%%%%%%%%%%%%%%%%%%%%%%%%%%%%%%%%%%%%%%%%%%%%%%%%%%%%%%%%%%%%%%%%%%%%%%%%%%%%%%%%%%%%%
\section{Proof of Theorem \ref{mp}}\label{proofofmt}
%%%%%%%%%%%%%%%%%%%%%%%%%%%%%%%%%%%%%%%%%%%%%%%%%%%%%%%%%%%%%%%%%%%%%%%%%%%%%%%%%%%%%%%%%%%%%

The Newton non-degenerate case being already treated in \cite[Chap.~III, Theorem (3.4)]{O3}, we only have to consider the situation where $f$ is Newton degenerate.
As $f$ is weighted homogeneous with respect to $P$, we have $\Gamma(f)=\Delta(P)$. Moreover, since $f$ has an isolated singularity at $\mathbf{0}$, the Newton boundary $\Gamma(f)$ meets all the coordinate $2$-planes. If $f$ is Newton degenerate on an 1-dimensional face $\Delta$ of $\Gamma(f)$, then, necessarily, $\Delta$ is contained in a coordinate $2$-plane.
Let us assume, for instance, that $f$ is Newton degenerate on the $1$-dimensional face $\Gamma(f)\cap\mathbb{R}^I=\Gamma(f^I)$ with $I:=\{2,3\}$. (Here, $\mathbb{R}^I$ is defined in a similar way as $\mathbb{C}^I$.) Note that $\Gamma(f^I)=\Delta(e_1)$, where $e_1:={}^t(1,0,0)$.
To simplify, we assume that $f$ is Newton non-degenerate on all the other $1$-dimensional faces (the argument is similar when $f$ degenerates on several edges).

We start with the choice of an admissible subdivision, that is, a regular simplicial cone subdivision of the dual Newton diagram $\Gamma^*(f)$. Here, let us recall that a simplicial cone generated by two primitive integral vectors $Q_1$ and $Q_2$ (hereafter denoted by $\mbox{Cone}(Q_1,Q_2)$) is called \emph{regular} if $\{Q_1,Q_2\}$ can be extended to an (integral) basis $\{Q_1,Q_2,Q_3\}$ of the lattice $\mathbb{Z}^3$. This is equivalent to saying that $\det(Q_1,Q_2)=1$ (see \cite[Chap.~II, \S 1]{O3}).

\begin{lemma}\label{lem-det}
$\mbox{\emph{Cone}}(P,e_1)$ is a regular simplicial cone of $\Gamma^*(f)$, that is, $\gcd(p_2,p_3)=1$.
\end{lemma}

\begin{proof}
As $\Delta(e_1)\subseteq\Delta(P)$, $\mbox{Cone}(P,e_1)$ is a (simplicial) cone of $\Gamma^*(f)$. We will show that $r:=\gcd(p_2,p_3)=1$. We argue by contradiction: suppose that $r\geq 2$. Then, since $P$ is primitive, we have $\gcd(p_1,r)=1$. As $d=d(P,f)=d(P,f^{I})$ and $d(P,f^{I})\equiv 0 \mbox{ mod } r$, the relation $\gcd(p_1,r)=1$ implies that $f$ has no monomial of the form $z_1z_2^{\alpha_2}z_3^{\alpha_3}$. Therefore there exists a polynomial $g$ such that $f$ is written as $f=f^{I}+z_1^2\, g$.
From this expression, we deduce that if $(z_{2},z_{3})$ is a critical point of $f^{I}$, then $(0,z_{2},z_{3})$ is a critical point of $f$. Now, factorize $f^{I}(0,z_2,z_3)$ as
\begin{equation}\label{facfzi}
f^{I}(0,z_2,z_3) = c\, z_2^az_3^b\prod_{i=1}^{\ell}(z_2^{p_3}-\lambda_i\, z_3^{p_2})^{k_i},
\end{equation}
where $c$ is a coefficient and $\lambda_1,\ldots,\lambda_\ell\in\mathbb{C}^*$ are mutually distinct. 
As $f^{I}$ is Newton degenerate, there exists at least one index $i$ with $k_i\geq 2$. For such an index $i$, the point $(\lambda_i^{1/p_3}u^{p_2}, u^{p_3})$ is a critical point of $f^I$ for any $u\in\mathbb{C}^*$. It follows that $(0,\lambda_i^{1/p_3}u^{p_2}, u^{p_3})$ is a critical point of $f$ for any $u\in\mathbb{C}^*$, and hence $f$ does not have an isolated singularity at~$\mathbf{0}$ \textemdash\ a contradiction.
\end{proof}

Combined with the \emph{subdivision principle} (see \cite[Chap.~II, Lemmas (2.1), (2.5) and (2.6)]{O3}), Lemma \ref{lem-det} shows that there exists a regular simplicial cone subdivision $\Sigma^*$ of $\Gamma^*(f)$ such that:
\begin{enumerate}
\item
$\mbox{Cone}(P,e_1)$ is a (regular) cone of $\Sigma^*$;
\item
there are two unique (regular) cones $\sigma:=\mbox{Cone}(P,e_1,R)$ and $\sigma':=\mbox{Cone}(P,e_1,R')$ of $\Sigma^*$ such that $\Delta(R)$ and $\Delta(R')$ are the endpoints of the $1$-dimensional face $\Delta(e_1)$. Let us assume, for instance, that $\Delta(R)=(0,a,b+p_2k)$ and $\Delta(R')=(0,a+p_3k,b)$, where $k:=\sum_{i=1}^{\ell}k_i$ and $a,b,k_i$ are the exponents that appear in \eqref{facfzi}.
\end{enumerate}
Let $\pi\colon X\to\mathbb{C}^3$ be the toric modification associated with such a subdivision $\Sigma^*$, and let $\mathbb{C}^3_\sigma$ and $\mathbb{C}^3_{\sigma'}$ be the toric charts of $X$ associated with the cones $\sigma$ and $\sigma'$, respectively.
Since the situation in these charts are similar, it suffices to consider one of them. Let us consider, for instance, the chart $\mathbb{C}^3_\sigma$.
As a matrix, $\sigma$ takes the form
\begin{equation*}
\sigma=
\begin{pmatrix}
p_1 & 1 & r_1\\
p_2 & 0 & r_2\\
p_3 & 0 & r_3
\end{pmatrix},
\end{equation*}
where ${^t(r_1,r_2,r_3)}:=R$. Since this matrix is regular, we have $\pm 1=\det(\sigma)=p_3r_2-p_2r_3$, and since $R$ takes its minimal value on $\Delta(R)=(0,a,b+p_2k)$, in fact we have $p_3r_2-p_2r_3=1$.
Expand $f_s(z_1,z_2,z_3)$ with respect to the variable $z_1$ as
\begin{equation}\label{dg1-bis}
f_s(z_1,z_2,z_3) = \sum_{j=0}^{[d/p_1]}z_1^j\, g_{s,j}(z_2,z_3),
\end{equation}
where $[d/p_1]$ denotes the integral part of $d/p_1$.
Note that $g_{s,0}$ is nothing else than the face function $f_s^I\equiv f_{s,\Delta(e_1)}$ associated with the face $\Delta(e_1)$. Also, observe that $g_{s,j}$ is a weighted homogeneous polynomial of weighted degree $d-jp_1$ with respect to the weights $(p_2,p_3)$.
In the toric chart $\mathbb{C}^3_\sigma$ with coordinates $\mathbf{y}=(y_1,y_2,y_3)$, the pull-back $\pi^*f_s$ of $f_s$ by $\pi$ is written as
\begin{equation}\label{eqpb1}
\begin{aligned}
\pi^*f_s(\mathbf{y}) 
= f_s(y_1^{p_1} y_2\, y_3^{r_1}, y_1^{p_2} y_3^{r_2}, y_1^{p_3} y_3^{r_3})
= y_1^{d} y_3^{d(R)}  
\sum_{j=0}^{[d/p_1]} y_2^{j} h_{s,j}(y_3),
\end{aligned}
\end{equation} 
where $h_{s,j}(y_3)$ is defined by  
\begin{equation}\label{rgh}
\pi^*(z_1^j g_{s,j})(\mathbf{y}) = y_1^{d} y_3^{d(R)} \times y_2^j h_{s,j}(y_3).
\end{equation} 
Note that the pull-back $\Pi^*F$ of $F$ by $\Pi$ is given by $\Pi^*F(s,\mathbf{y})=\pi^*f_s(\mathbf{y})$, and the strict transform $\widetilde{V}(F)$ of $V(F)$ by $\Pi$ is defined by the polynomial
\begin{equation}\label{dpstF}
\tilde F(s,\mathbf{y}):=\Pi^*F(s,\mathbf{y})\big/\big(y_1^{d} y_3^{d(R)}\big)=\sum_{j=0}^{[d/p_1]} y_2^{j} h_{s,j}(y_3).
\end{equation}
The defining polynomial $\tilde f_s(\mathbf{y})$ of the strict transform $\widetilde{V}(f_s)$ of $V(f_s)$ by $\pi$ identifies with the restriction of $\tilde F$ to $\{s\}\times \mathbb{C}^3_\sigma$.

To prove the theorem, we have to show that $\widetilde{V}(F)$ (respectively, $\widetilde{V}(f_s)$ for all small~$s$) is smooth and transverse to the exceptional divisors $D_\varepsilon(0)\times \hat E(Q)$ (respectively, $\hat E(Q)$) for $\varepsilon$ small enough.\footnote[2]{Note that $U\times \hat E(Q)$ and $\hat E(Q)$ are exceptional divisors if and only if $d(Q)>0$. In particular, $U\times \hat E(e_1)$ and $\hat E(e_1)$ are not exceptional divisors.} Here, by $\hat E(Q)$, we mean the divisor of $\pi$ corresponding to the vertex $Q$ (see \cite[Chap.~II, \S 1]{O3} or \cite[\S 2.5.1]{O1} for the definition).
The polynomial $h_{0,0}(y_3)$ (which appears for $s=j=0$ in the expression \eqref{eqpb1}) is given by \begin{equation}\label{a512}
h_{0,0}(y_3)=\prod_{i=1}^{\ell} (y_3-\lambda_i)^{k_i}. 
\end{equation}
Indeed, by \eqref{rgh}, $y_1^{d} y_3^{d(R)}h_{0,0}(y_3)=\pi^*(g_{0,0})(\mathbf{y}) = \pi^*f^I_0(\mathbf{y})$. Now, as $\Delta(R)=(0,a,b+p_2k)$, we have $d(R)=ar_2+br_3+kp_2r_3$, and since $f^I_0$ is weighted homogeneous of weighted degree $d$, we also have $p_2a+p_3(b+kp_2)=d$. Then, from \eqref{facfzi}, we deduce that
\begin{align*}
\pi^*f^I_0(\mathbf{y}) = y_1^{ap_2+bp_3+kp_2p_3} y_3^{ar_2+br_3+kp_2r_3}\prod_{i=1}^{\ell} (y_3-\lambda_i)^{k_i}= y_1^{d} y_3^{d(R)}\prod_{i=1}^{\ell} (y_3-\lambda_i)^{k_i}.
\end{align*} 
On the other hand, for $s\not=0$, the Newton non-degeneracy of $f_s^I$ implies that $h_{s,0}(y_3)$ has $k$ simple roots (namely, $k_i$ roots converging to $\lambda_i$ as $s\to 0$ for each $1\leq i\leq \ell$).
Pick an arbitrary small positive number $\delta$, and consider the open disc $D_\delta(\lambda_i)$ with centre $\lambda_i$ and radius $\delta$ in the $y_3$-coordinate. By taking $s$ in a sufficiently small disc $D_\varepsilon(0)$, we may assume that all the roots of $h_{s,0}(y_3)$ are contained in $D:=\bigcup_{i=1}^\ell D_\delta(\lambda_i)$ ($k_i$ roots in each disc $D_\delta(\lambda_i)$). 

\begin{lemma}\label{rem}
For any index $i$ such that $k_i\geq 2$, we have $h_{0,1}(\lambda_i)\not=0$.
\end{lemma}

\begin{proof}
This is an immediate consequence of \eqref{rgh} and the fact that in a small neighbourhood of $\mathbf{0}\in\mathbb{C}^I$, the set $A_i$ defined for any index $i$ with $k_i\geq 2$ by 
\begin{equation*}
A_i:=\{(0,z_2,z_3)\in\mathbb{C}^I\setminus\{(0,0,0)\}\, ;\, z_2^{p_3}-\lambda_i\, z_3^{p_2}=g_{0,1}(z_2,z_3)=0\}
\end{equation*}
is empty.\footnote{Similarly, note that for any $z_2,z_3$ near zero, we also have: 
\begin{enumerate}
\item 
$g_{0,1}(0,z_3)\not=0$ if $a\geq 2$; 
\item 
$g_{0,1}(z_2,0)\not=0$ if $b\geq 2$.
\end{enumerate}} 
This vacuity, in turn, follows from the fact that $f$ has an isolated singularity at~$\mathbf{0}$. Indeed, as in the proof of Lemma \ref{lem-det}, for any $i$ such that $k_i\geq 2$, the point $(\lambda_i^{1/p_3}u^{p_2}, u^{p_3})$ is a critical point of $f^I$ for any $u\in\mathbb{C}^*$. As this is not a critical point of $f=f_0$, we must have 
\begin{equation*}
0\not=\frac{\partial f_0}{\partial z_1}(0,\lambda_i^{1/p_3}u^{p_2}, u^{p_3})=g_{0,1}(\lambda_i^{1/p_3}u^{p_2}, u^{p_3}),
\end{equation*}
which implies the vacuity of $A_i$ near $\mathbf{0}\in\mathbb{C}^I$.
\end{proof}

Lemma \ref{rem} implies that if $\delta$ and $\varepsilon$ are chosen sufficiently small, then $h_{s,1}(y_3)\not=0$  for any $(s,y_3)\in D_\varepsilon(0)\times D'$, where $D'$ is the union of those $D_\delta(\lambda_i)$'s for which $k_i\geq 2$. Thus, since
\begin{equation*}
\tilde F(s,\mathbf{y})\equiv h_{s,0}(y_3)+y_2h_{s,1}(y_3) \ \mbox{modulo}\ y_2^2
\end{equation*} 
(see \eqref{dpstF}), it follows that if $D_\eta(0)$ is a sufficiently small open disc at $0$ in the $y_2$-coordinate, then we have
\begin{equation}\label{nzd}
\frac{\partial \tilde F}{\partial y_2} (s,\mathbf{y})\not=0
\end{equation} 
for any $(s,y_2,y_3)\in \mathscr{D}:=D_\varepsilon(0)\times D_\eta(0)\times D'$. This implies that $\widetilde{V}(F)$ is non-singular at any point of $\mathscr{D}$. Similarly, $\widetilde{V}(f_s)$, $s\in D_\varepsilon(0)$, is non-singular at any point of $D_\eta(0)\times D'$. The relation \eqref{nzd} also implies that the strict transform $\widetilde{V}(F)$ cannot be tangent to an exceptional divisor $D_\varepsilon(0)\times \hat E(Q)$ at a point in $\mathscr{D}$. It can only be tangent to $D_\varepsilon(0)\times \hat E(e_1)$ (which is not an exceptional divisor) and this can only happen at a point of the form $(0,0,\lambda_i)$, where $i$ is such that $k_i\geq 2$. Similarly for $\widetilde{V}(f_s)$, $s\in D_\varepsilon(0)$. (Note that by our choice of $D_\eta(0)$ and $D$, we have $\widetilde{V}(F)\cap (D_\varepsilon(0)\times \hat E(e_1))\subseteq D_\varepsilon(0)\times D_\eta(0)\times D$.)
Now, outside of $\mathscr{D}$, the smoothness and transversality of $\widetilde{V}(F)$ to the exceptional divisors follow from the classical argument developed in the proof of \cite[Chap.~III, Theorem (3.4)]{O3} which concerns the Newton non-degenerate case. Similarly for $\widetilde{V}(f_s)$, $s\in D_\varepsilon(0)$.

\begin{remark}
In the chart $(\mathbb{C}_{\sigma'},(y_1',y_2',y_3'))$ corresponding to $\sigma':=\mbox{Cone}(P,e_1,R')$, the counterpart of the factor $(y_3-\lambda_i)^{k_i}$ that appears in \eqref{a512} is given by $(1-\lambda_iy_3')^{k_i}$, so that the situation is similar in both charts. 
\end{remark}

\begin{remark}\label{rem-ltf}
The above proof shows that $\hat E(P)\cap\widetilde{V}(f_s)$ and $\hat E(P)\cap\widetilde{V}(f_0)$ are diffeomorphic smooth manifolds for all small $s$.
\end{remark}

%%%%%%%%%%%%%%%%%%%%%%%%%%%%%%%%%%%%%%%%%%%%%%%%%%%%%%%%%%%%%%%%%%%%%%%%%%%%%%%%%%%%%%%%%%%%
\section{Invariance of the monodromy zeta-function on $\mathcal{W}(P,d,\Gamma)$}\label{izf}
%%%%%%%%%%%%%%%%%%%%%%%%%%%%%%%%%%%%%%%%%%%%%%%%%%%%%%%%%%%%%%%%%%%%%%%%%%%%%%%%%%%%%%%%%%%%

We still assume that $n=3$. As an application of Theorem \ref{mp}, we give in this section a new geometrical proof of the following weak version of Milnor--Orlik's theorem: \emph{``For any $f, g \in \mathcal{W}(P,d,\Gamma)$, the monodromy zeta-functions $\zeta_{f,\mathbf{0}}(t)$ and $\zeta_{g,\mathbf{0}}(t)$ are identical''}. We recall that $\mathcal{W}(P,d,\Gamma)$ is the set of all (possibly Newton degenerate) polynomials $f\in \mathcal{W}(P,d)$ such that $\Gamma(f)=\Gamma$. We also recall that if we drop the weighted homogeneity assumption, then two polynomials with the same Newton boundary may have different monodromy zeta-functions if one of them is Newton degenerate.

Now, to show the assertion, it suffices to prove that if $f\in\mathcal{W}(P,d,\Gamma)$ and $\{f_s\}$ is a generic deformation of its coefficients such that $f_s$ is Newton non-degenerate for $s\not=0$, then the monodromy zeta-function $\zeta_{f_s,\mathbf{0}}(t)$ is independent of $s$ for all small $s$. Indeed, by a theorem of Varchenko \cite[Theorem (4.1)]{V}, we know that the monodromy zeta-function of a Newton non-degenerate function depends only on its Newton boundary.  

To simplify, as in the proof of Theorem \ref{mp}, we shall assume that $f$ is Newton degenerate only on the $1$-dimensional face $\Gamma(f)\cap \mathbb{R}^{\{2,3\}}=\Delta(e_1,f)$.
Let $\Sigma^*$ be a regular simplicial cone subdivision of $\Gamma^*(f)$ as in (the proof of) Theorem \ref{mp}, and let $\pi\colon X\to\mathbb{C}^3$ be the corresponding toric modification. By the theorem, $\pi$ is a simultaneous good resolution of the family of functions $f_s$ for all small $s$.
Put 
\begin{equation*}
E(Q,f_s):=\hat E(Q)\cap\widetilde{V}(f_s).
\end{equation*}
Then by the A'Campo formula for the monodromy zeta-function \cite[Th\'eor\`eme 3]{A}, we have
\begin{equation*}
\zeta_{f_s,\mathbf{0}}(t)=\prod_{Q\in\mathcal{V}_+(f_s)}(1-t^{d(Q,f_s)})^{-\chi(\hat E(Q,f_s)'')}
\end{equation*}
for all small $s$. Here, $\mathcal{V}_+(f_s)$ is the set of vertices $Q\in\Sigma^*$ such that $d(Q,f_s)>0$, and $\hat E(Q,f_s)''$ is defined by
\begin{equation*}%\label{rdm2}
\hat E(Q,f_s)'':=\Bigg(\hat E(Q)\Bigg\backslash \Bigg(\widetilde{V}(f_s)\cup\bigcup_{\substack{S\in\mathcal{V}_+(f_s)\\ S\not=Q}} \hat E(S)\Bigg)\Bigg)\cap\pi^{-1}(\mathbf{0}).
\end{equation*}
The notation $\chi(\hat E(Q,f_s)'')$ stands for the Euler--Poincar\'e characteristic of $\hat E(Q,f_s)''$. Note that $f_s$ being obtained from a deformation of the coefficients of $f$, the set $\mathcal{V}_+(f_s)$ is independent of $s$. 
Now, as the degenerate face of $\Gamma(f)=\Delta(P)$ only concerns the boundary of $\Delta(P)$, to prove the assertion we only have to consider $\hat E(P,f_s)''$ and show the following lemma. 
(For all the other terms $\hat E(Q,f_s)''$, $Q\not=P$, the Euler--Poincar\'e characteristic is independent of $s$ by the classical argument developed in the proofs of \cite[Chap.~III, Theorem (5.3)]{O3} or \cite[Theorem (4.1)]{V}, which correspond to the Newton non-degenerate case.)

\begin{lemma}\label{conttmzf}
For any $s$ small enough, $\chi(\hat E(P,f_s)'')=\chi(\hat E(P,f_0)'')$.
\end{lemma}

\begin{proof}
Essentially this follows from the argument developed in the proof of Theorem \ref{mp}, which shows that $E(P,f_0)$ and $E(P,f_s)$ are diffeomorphic smooth manifolds for all small $s$ (see Remark \ref{rem-ltf}), so that $\chi(E(P,f_s))=\chi(E(P,f_0))$.
We shall also use the canonical toric stratification of $E(P,f_s)$ introduced in \cite[Chap.~II, \S 1 and Chap.~III, \S 5]{O3} and which is defined as follows. Let $\mathcal{C}(P)$ be the set of cones of $\Sigma^*$ having $P$ as a vertex, and let $\mathcal{V}(P)$ be the set of vertices of $\Sigma^*$ which are \emph{adjacent} to $P$.\footnote[8]{Here, we say that a vertex $Q\in\Sigma^*$ is \emph{adjacent} to $P$ if $\mbox{Cone}(P,Q)$ is a cone of $\Sigma^*$ (see \cite[\S 2.5.1]{O1}).}
For any $\sigma\in\mathcal{C}(P)$, put
\begin{equation*}
\hat E(\sigma)^*:=\bigg(\bigcap_{Q\in\mathcal{V}(\sigma)}\hat E(Q)\bigg)\bigg\backslash\bigg(\bigcup_{Q'\in\mathcal{V}(P)\setminus \mathcal{V}(\sigma)}\hat E(Q')\bigg)
\quad\mbox{and}\quad
E(\sigma,f_s)^*:=\hat E(\sigma)^*\cap\widetilde{V}(f_s),
\end{equation*}
where $\mathcal{V}(\sigma)$ is the set of vertices of $\sigma$. Note that $\hat E(\sigma)^*$ is isomorphic to $\mathbb{C}^{*(3-k)}$, where $k$ is the number of vertices of $\sigma$. Then the \emph{canonical toric stratification} of $E(P,f_s)$ is defined by the partition
\begin{equation*}%\label{cs}
E(P,f_s) = \bigsqcup_{\sigma\in\mathcal{C}(P)}E(\sigma,f_s)^*.
\end{equation*}
Let $\mathcal{C}(P)'\subseteq\mathcal{C}(P)$ be the set of cones whose all vertices, except $P$, do not belong to $\mathcal{V}_+(f_s)=\mathcal{V}_+(f_0)$, and let $\mathcal{C}(P)'':=\mathcal{C}(P)\setminus \mathcal{C}(P)'$. Then 
\begin{equation*}
\hat E(P,f_s)''= \bigsqcup_{\sigma\in\mathcal{C}(P)'}\hat E(\sigma)^* \setminus E(\sigma,f_s)^*,
\end{equation*}
while $\bigsqcup_{\sigma\in\mathcal{C}(P)''} E(\sigma,f_s)^*$ is independent of $s$.
Now, as the sums
\begin{equation*}
\chi(E(P,f_s))=\sum_{\sigma\in\mathcal{C}(P)'}\chi(E(\sigma,f_s)^*)+\sum_{\sigma\in\mathcal{C}(P)''} \chi(E(\sigma,f_s)^*)
\end{equation*}
and $\sum_{\sigma\in\mathcal{C}(P)''} \chi(E(\sigma,f_s)^*)$ are independent of $s$, so is  $\sum_{\sigma\in\mathcal{C}(P)'}\chi(E(\sigma,f_s)^*)$. It follows that 
\begin{equation*}
\chi(\hat E(P,f_s)'')=\sum_{\sigma\in\mathcal{C}(P)'}\chi(\hat E(\sigma)^*)-\chi(E(\sigma,f_s)^*)
\end{equation*}
is independent of $s$ as well. (Note that the divisor $\hat E(P)$ being compact, we do not need to take the intersection with $\pi^{-1}(\mathbf{0})$ in the above argument.)
\end{proof}

\section*{Acknowledgments} This research was supported by the Narodowe Centrum Nauki under the grant number 2023/49/B/ST1/00848.

\bibliographystyle{amsplain}

\end{document}